\documentclass[12pt,a4paper]{amsart}

\usepackage[T1]{fontenc}
\usepackage[utf8]{inputenc}
\usepackage[left=25mm,top=35mm,right=25mm,bottom=26mm,headsep=10mm]{geometry}
\usepackage{microtype,textcomp,fbb}
\usepackage{
    amsmath,  amssymb,  amsthm,   amscd,
    gensymb,  graphicx, etoolbox, 
    booktabs, stackrel, mathtools    
}
\usepackage{bm}
\usepackage{caption}
\usepackage[usenames,dvipsnames]{xcolor}
\usepackage[colorlinks=true, linkcolor=blue, citecolor=blue, urlcolor=blue, breaklinks=true]{hyperref}
\usepackage[capitalise]{cleveref}

\usepackage{tikz, calc}
\usetikzlibrary{automata}
\usepackage{graphicx}
\usepackage[nomessages]{fp}
\usepackage{xparse}
\ExplSyntaxOn
\DeclareExpandableDocumentCommand{\floor}{m}
 {
  \fp_eval:n { floor ( #1 ) }
 }
\ExplSyntaxOff

\newtheorem{theorem}{Theorem}[section]

\newtheorem{corollary}[theorem] {Corollary}
\newtheorem{definition}[theorem]{Definition}
\newtheorem{example}[theorem]{Example}
\newtheorem{lemma}[theorem]{Lemma}

\newtheorem{proposition}[theorem]{Proposition}
\newtheorem{remark}[theorem]{Remark}
\newtheorem{result}[theorem]{Result}
\newtheorem{claim}[theorem]{Claim}

\usepackage{tikz}
\usetikzlibrary{arrows}
\usetikzlibrary{shapes}
\usepackage{float}
\usepackage{tabularx}
\usepackage{kantlipsum}
\tikzset{edgee/.style = {> = latex'}}
\usepackage{array}
\newcolumntype{P}[1]{>{\centering\arraybackslash}p{#1}}
\newcolumntype{M}[1]{>{\centering\arraybackslash}m{#1}}

\newcommand\R{\mathbb{R}}
\newcommand\Z{\mathbb{Z}}

\newcommand{\A}{\mathcal{A}}
\newcommand{\B}{\mathcal{B}}
\newcommand{\C}{\mathcal{C}}
\newcommand{\Hy}{\mathcal{H}}

\newcommand{\ipa}{\mathrm{L}}
\newcommand{\lt}[2]{\ensuremath{\alpha_{#1}^{(#2)}}}

\begin{document}

\title[Refinements of the braid arrangement]{Refinements of the braid arrangement and two parameter Fuss-Catalan numbers}
\author{Priyavrat Deshpande}
\address{Chennai Mathematical Institute}
\email{pdeshpande@cmi.ac.in}
\author{Krishna Menon}
\address{Chennai Mathematical Institute}
\email{krishnamenon@cmi.ac.in}
\author{Writika Sarkar}
\address{Chennai Mathematical Institute}
\email{writika@cmi.ac.in}
\thanks{PD and KM are partially supported by a grant from the Infosys Foundation}

\begin{abstract}

A hyperplane arrangement in $\mathbb{R}^n$ is a finite collection of affine hyperplanes.
Counting regions of hyperplane arrangements is an active research direction in enumerative combinatorics.
In this paper, we consider the arrangement $\mathcal{A}_n^{(m)}$ in $\mathbb{R}^n$ given by $\{x_i=0 \mid i \in [n]\} \cup \{x_i=a^kx_j \mid k \in [-m,m], 1\leq i<j \leq n\}$ for some fixed $a>1$.
It turns out that this family of arrangements is closely related to the well-studied extended Catalan arrangement of type $A$.
We prove that the number of regions of $\mathcal{A}_n^{(m)}$ is a certain generalization of Catalan numbers called two parameter Fuss-Catalan numbers. 
We then exhibit a bijection between these regions and certain decorated Dyck paths.
We also compute the characteristic polynomial and give a combinatorial interpretation for its coefficients.
Most of our results also generalize to sub-arrangements of $\mathcal{A}_n^{(m)}$ by relating them to deformations of the braid arrangement.
\end{abstract}

\maketitle


\section{Introduction}\label{sec1}
A \emph{hyperplane arrangement} $\A$ is a finite collection of affine hyperplanes (i.e., codimension $1$ subspaces and their translates) in $\R^n$. 
Without loss of generality we assume that arrangements we consider are \emph{essential}, i.e., the subspace spanned by the normal vectors is the ambient vector space.   
A \emph{flat} of $\A$ is a nonempty intersection of some of the hyperplanes in $\A$; the ambient vector space is a flat since it is an intersection of no hyperplanes. 
Flats are naturally ordered by reverse set inclusion; the resulting poset is called the \emph{intersection poset} and is denoted by $\ipa(\A)$. 
A \emph{region} of $\A$ is a connected component of $\R^n\setminus \bigcup \A$. 
Counting the number of regions of arrangements using diverse combinatorial methods is an active area of research. 

The \emph{characteristic polynomial} of $\A$ is defined as 
\[\chi_{\A}(t) := \sum_{x\in\ipa(\A)} \mu(\hat{0},x)\, t^{\dim(x)}\]
where $\mu$ is the M\"obius function of the intersection poset and $\hat{0}$ corresponds to the flat $\R^n$. 
Using the fact that every interval of the intersection poset of an arrangement is a geometric lattice, we have
\begin{equation}\label{charform}
    \chi_\A(t) = \sum_{i=0}^n (-1)^{n-i} c_i t^i
\end{equation}
where $c_i$ is a non-negative integer for all $0 \leq i \leq n$ \cite[Corollary 3.4]{stanarr07}.
The characteristic polynomial is a fundamental combinatorial and topological invariant of the arrangement and plays a significant role throughout the theory of hyperplane arrangements.

In this article, our focus is on the enumerative aspects of (rational) arrangements in $\R^n$. 
In that direction we have the following seminal result by Zaslavsky.

\begin{theorem}[\cite{zas75}]\label{zaslavsky}
Let $\A$ be an arrangement in $\R^n$. Then the number of regions of $\A$ is given by 
\begin{align*}
   r(\A) &= (-1)^n \chi_{\A}(-1) \\
         &= \sum_{i=0}^n c_i. 
\end{align*}
\end{theorem}

The finite field method,  developed by Athanasiadis \cite{athanas96}, converts the computation of the characteristic polynomial to a point counting problem. 
A combination of these two results allowed for the computation of the number of regions of several arrangements of interest. 

Another way to count the number of regions is to give a bijective proof. 
This approach involves finding a combinatorially defined set whose elements are in bijection with the regions of the given arrangement and are easier to count. 
For example, the \textit{braid arrangement} in $\R^n$ is given by
\begin{equation*}
    \{x_i-x_j=0 \mid 1 \leq i < j \leq n\}.
\end{equation*}
It is straightforward to verify that its regions correspond to the permutations of $[n]$.
Hence the number of regions of the braid arrangement in $\R^n$ is $n!$.

Another important class of arrangements is the extended Catalan arrangements.
For any $m,n \geq 1$, the $m$-Catalan arrangement in $\R^n$ is given by
\begin{equation*}
    \C_n^{(m)} := \{x_i-x_j=k \mid k \in [-m,m], 1 \leq i < j \leq n\}.
\end{equation*}
These arrangements and their sub-arrangements have been studied in great detail by Bernardi in \cite{ber}.
It is well-known that the number of regions of $\C_n^{(m)}$ where $x_1<x_2<\cdots<x_n$ (also known as the dominant regions) is given by
\begin{equation*}
    A_n(m,1):=\frac{1}{n(m+1)+1}\binom{n(m+1)+1}{n}.
\end{equation*}
Using this, it is easy to see that $r(\C_n^{(m)})=n! \cdot A_n(m,1)$.
The numbers $A_n(m,1)$ are called the extended Catalan numbers or the Fuss-Catalan numbers.

The two parameter Fuss-Catalan numbers or the Raney numbers are given by
\begin{equation*}
    A_n(m,r) := \frac{r}{n(m+1)+r}\binom{n(m+1)+r}{n}
\end{equation*}
for all positive integers $n,m,r$.
These numbers were first studied by Raney in \cite{raney}.

Just as the extended Catalan arrangement works for $r=1$, one could ask if, for each $r > 1$, there is a nice family of arrangements whose regions are given by multiples of $A_n(m,r)$, preferably $n! \cdot A_n(m,r)$.
In the present article we study a family of arrangements that works for $r=2$.
For any $m,n \geq 1$, the arrangement in $\R^n$ given by
\begin{equation}\label{arrdef}
    \A_n^{(m)}:=\{x_i=0 \mid i \in [n]\} \cup \{x_i=a^kx_j \mid k \in [-m,m], 1\leq i<j \leq n\}
\end{equation}
where $a>1$ is a fixed real number.
In \Cref{sec2}, we show that $a$ does not affect the combinatorics of the arrangement.
The arrangement $\A_n^{(m)}$ is a generalization of the arrangement $\alpha_n$ (which is $\A_n^{(1)}$ with $a=2$) defined in \cite{seo}.

The usual Catalan numbers are given by $A_n(1,1)$. 
They occur in various counting problems; in fact, a vast variety of combinatorial objects are counted by these numbers. 
For example, the reader can look at the book \cite{stancat15} by Stanley which contains more than 200 interpretations of Catalan numbers. 
The extended Catalan numbers also count many interesting objects. 
Other than the regions of $\C_n^{(m)}$, the following objects are also counted by $n!\cdot A_n(m, 1)$:
\begin{enumerate}
    \item Certain non-nesting partitions (follows from \cite[Theorem 2.2]{athanasiadisnesting}) or an equivalent formulation called $m$-sketches (introduced by Bernardi \cite[Section 8]{ber}). 
    \item Labeled $m$-ary trees (Bernardi \cite{ber} and Stanley \cite[part (b) of A14]{stancat15}).
    \item Generalized Dyck paths (\cite[part (c) of A14]{stancat15}).
\end{enumerate}
There are well-known bijections between the regions of $\C_n^{(m)}$ and these combinatorial objects (for example, see \cite[Section 8.1]{ber}). 
An outcome of our work is a generalization of these bijections. 
In particular, we introduce analogs of above objects which correspond to the regions of $\A_n^{(m)}$. 

When studying an arrangement, another interesting question is whether the coefficients of its characteristic polynomial can be combinatorially interpreted.
By \Cref{zaslavsky}, we know that the sum of the absolute values of the coefficients is the number of regions.
Hence, one could ask if there is a statistic on the regions whose distribution is given by the coefficients of the characteristic polynomial.
The characteristic polynomial of the braid arrangement in $\R^n$ is 
$t(t-1)\cdots(t-n+1)$ \cite[Corollary 2.2]{stanarr07}. 
Hence, the coefficients are the Stirling numbers of the first kind.
Consequently, the distribution of the statistic `number of cycles' on the set of permutations of $[n]$ (which correspond to the regions of the arrangement) is given by the coefficients of the characteristic polynomial.

The main aim of this paper is to study the arrangement $\A_n^{(m)}$ and its sub-arrangements.
In \Cref{sec2}, we compute the characteristic polynomial of $\A_n^{(m)}$.
We also relate the characteristic polynomial of $\A_n^{(m)}$ to that of $\C_n^{(m)}$.
In \Cref{sec3}, we describe certain combinatorial objects called decorated Dyck paths and decorated non-nesting partitions and show that they are in bijection with the regions of $\A_n^{(m)}$.
Finally in \Cref{sec4}, we describe a statistic on decorated Dyck paths whose distribution is given by the coefficients of the characteristic polynomial of $\A_n^{(m)}$.
In all sections, we also show that the results generalize to several sub-arrangements of $\A_n^{(m)}$.

\section{The Characteristic Polynomial}\label{sec2}

We begin by stating the finite field method developed by Athanasiadis \cite{athanas96}.

\begin{theorem}[{\cite[Theorem 2.2]{athanas96}}]\label{ffm} Let $\A$ be an arrangement in $\R^n$ defined over the integers and $q$ be a large enough prime number. Then 
\[ \chi_{\A}(q) = \#\left( \Z_q^n \setminus V_{\A} \right)\]
where $V_{\A}$ is the union of the hyperplanes in $\Z_q^n$ obtained by reducing $\A$ mod $q$.

\end{theorem}

We will use this result to find the characteristic polynomial for the arrangement $\A_n^{(m)}$. First we show that the actual value of $a$ in \eqref{arrdef} does not affect the combinatorics of the arrangement $\A_n^{(m)}$ by showing that $a$ does not affect the order of the intersection poset $\ipa(\A_n^{(m)})$.

\begin{proposition}\label{nodepa} The Hasse diagram for the intersection poset $\ipa(\A_n^{(m)})$ is independent of $a$.
\end{proposition}
\begin{proof}
We completely describe the intersection poset $\ipa(\A_n^{(m)})$ in terms of directed graphs.
This will show that the poset does not depend on the value of $a>1$.
The following proof is similar to that of \cite[Lemma 7.1]{ber}.

Let $\Hy:=\{H_1,\ldots,H_l\} \subseteq \A_n^{(m)}$.
We represent the intersection $\cap \Hy$, i.e., $H_1 \cap \cdots \cap H_l$ as a directed multi-graph on the vertex set $[n]$ which we call $G(\Hy)$.
If $x_i=a^kx_j$ for some distinct $i,j \in [n]$ and $k \in [m]$ is a hyperplane in $\Hy$, then an edge labeled $k$ is drawn from vertex $j$ to vertex $i$ representing the fact the multiplying $a^k$ to $x_j$ gives $x_i$.
Similarly, an edge labeled $-k$ is drawn from vertex $i$ to vertex $j$.
If $x_i=0$ then a loop is drawn on the vertex $i$.
After this is done for all hyperplanes in $\Hy$, if there is a connected component in the directed graph which has a vertex with a loop, then loops are given to all the vertices in the component and all other edges are deleted.

\begin{example}\label{nodepaex}
Let $\Hy \subseteq \A_6^{(4)}$ have the hyperplanes $x_1=0$, $x_1=a^2x_2$, $x_4=ax_3$, $x_5=a^3x_4$, and $x_5=a^4x_3$.
Then $G(\Hy)$ is shown is \Cref{ghex}.
\end{example}

\begin{figure}[H]
    \centering
    \begin{tikzpicture}
    \node (6) [circle,draw=black,inner sep=2pt] at (-3.5,-1.5) {$6$};
    \node (4) [circle,draw=black,inner sep=2pt] at (-1,0) {$4$};
    \node (3) [circle,draw=black,inner sep=2pt] at (3,0) {$3$};
    \node (2) [circle,draw=black,inner sep=2pt] at (5.5,-0.75) {$2$};
    \node (1) [circle,draw=black,inner sep=2pt] at (5.5,-2.25) {$1$};
    \node (5) [circle,draw=black,inner sep=2pt] at (1,-3) {$5$};
    \draw [<-,edgee] (4) to [bend left=15pt] node[above] {\tiny $1$} (3);
    \draw [<-,edgee] (3) to [bend left=15pt] node[above] {\tiny $-1$} (4);
    \draw [<-,edgee] (5) to [bend left=15pt] node[right] {\tiny $4$} (3);
    \draw [<-,edgee] (3) to [bend left=15pt] node[right] {\tiny $-4$} (5);
    \draw [<-,edgee] (4) to [bend left=15pt] node[above] {\tiny $-3$} (5);
    \draw [<-,edgee] (5) to [bend left=15pt] node[above] {\tiny $3$} (4);
    \path (2) edge [loop above] (2);
    \path (1) edge [loop above] (1);
    \end{tikzpicture}
    \caption{$G(\Hy)$ for $\Hy$ given in \Cref{nodepaex}.}
    \label{ghex}
\end{figure}

\begin{claim}\label{cl1}
$\cap \Hy$ will be non-empty if and only if any two directed paths between vertices $i$ and $j$ in $G(\Hy)$ have labels with the same sum for all distinct $i,j \in [n]$.
In this case, we call $G(\Hy)$ \textit{consistent}.
\end{claim}

If $G(\Hy)$ is consistent, a point in the intersection can be constructed as follows.
First, all coordinates whose corresponding vertices have loops are set to zero.
For any component without a loop, fix an arbitrary value, say $c$, for some coordinate, say $x_i$, in the component.
For any other coordinate $x_j$ in the same component, consider a path from vertex $i$ to $j$ and suppose the sum of the labels on this path is $k$.
Then set $x_j=a^kc$.
By the property satisfied by $G(\Hy)$, this gives a point in $\cap \Hy$.
It is also clear that if $\cap \Hy$ is non-empty, then $G(\Hy)$ should be consistent.
This also shows that the dimension of $\cap \Hy$ in such a case is the number of components of $G(\Hy)$ that are not loops.

From the above discussion, we have the following characterization of the poset structure of $\ipa(\A_n^{(m)})$.
\begin{claim}\label{cl2}
If $G(\Hy_1)$ and $G(\Hy_2)$ are consistent, then $\cap \Hy_1 \subseteq \cap \Hy_2$ if and only if
\begin{enumerate}
    \item all looped vertices in $G(\Hy_2)$ are looped in $G(\Hy_1)$, and 
    \item for any distinct $i,j \in [n]$, if a path from $i$ to $j$ in $G(\Hy_2)$ has labels whose sum is $k$, then so does any path from $i$ to $j$ in $G(\Hy_1)$.
\end{enumerate}
\end{claim}

It now follows from \Cref{cl1} and \Cref{cl2} that $\ipa(\A_n^{(m)})$ does not depend on the value of $a$.
\end{proof}

Now we have the liberty to choose any $a$ convenient for us. To do so, we will take help of a result from number theory by Heath-Brown \cite{heath_brown}, which is a partial work towards a solution to Artin's conjecture on primitive roots, which states that every integer $k$, except $-1$ and a perfect square, is a primitive root modulo infinitely many primes. The main theorem of interest is stated below.

\begin{theorem}[{\cite[Corollary 2]{heath_brown}}]\label{primrootexistence}
There are at most three square-free positive integers for which Artin's conjecture fails.
\end{theorem}

In particular, this theorem guarantees the existence of a positive integer which is a primitive root modulo infinitely many primes. Hence we choose $a$ to be such an integer for our purpose. Now we are ready to find the characteristic polynomial for the arrangement $\A_n^{(m)}$.

\begin{proposition}\label{charpolyref} The characteristic polynomial of $\A_n^{(m)}$ is given by 
\[ \chi_{\A_n^{(m)}}(t) = (t-1)(t-mn-2)(t-mn-3)\dots (t-mn-n). \]

\end{proposition}

\begin{proof} Let $q>n$ be a large enough prime such that $a$ is a primitive root modulo $q$. By \Cref{ffm}, we know that $\chi_{\A_n^{(m)}}(q)=|S|$ where 
\[ S=\{ (x_1,\dots, x_n)\in \Z_q^n \mid x_i\neq 0, x_i\neq a^kx_j \text{ for all } i,j\in [n], i\neq j \text{ and } k\in [0,m]\}.\] 
Since $a$ is a generator for $\Z_q^\times = \{1,2,\dots, q-1\}$, we see that $1,a,a^2,\dots, a^{q-2}$ is a permutation of the elements $1,2,\dots, q-1$. Now we consider $q-1$ columns with $m+1$ entries each such that the $k$th column from the left contains the numbers $a^{k-1},a^k,\dots, a^{k+m-1}$ (where the exponents are taking mod $q-1$) as follows.

\[\begin{tabular}{ccccc}

$a^0$ & $a^1$ & $a^2$ & $\dots$ & $a^{q-2}$\\ 

$a^1$ & $a^2$ & $a^3$ & $\dots$ & $a^{q-1}$\\ 

$a^2$ & $a^3$ & $a^4$ & $\dots$ & $a^q$\\ 

$\vdots$ & $\vdots$ & $\vdots$ & $\dots$ & $\vdots$\\ 

$a^m$ & $a^{m+1}$ & $a^{m+2}$ & $\dots$ & $a^{m+q-2}$\\
\end{tabular}\]

If we call the topmost row the $0$th row, then note that 
\begin{itemize}
\item the $\ell$th row is $a^\ell$ times the corresponding numbers in the $0$th row,
\item the exponents in row $i$ is a cyclic shift of the exponents in row $i-1$ to the left by 1 step for $i>0$ when taken mod $q-1$, and hence row $i$ itself is a cyclic shift to the left by 1 step of row $i-1$.
\end{itemize}

Now, keeping the order the same, we arrange these columns on the circumference of a circle. 
As an illustration of the circular arrangement, consider $m=3, q=11, a=2$. The columns are

\[\begin{tabular}{cccccccccc}

$2^0$ & $2^1$ & $2^2$ & $2^3$ & $2^4$ & $2^5$ & $2^6$ & $2^7$ & $2^8$ & $2^9$ \\ 

$2^1$ & $2^2$ & $2^3$ & $2^4$ & $2^5$ & $2^6$ & $2^7$ & $2^8$ & $2^9$ & $2^0$ \\ 

$2^2$ & $2^3$ & $2^4$ & $2^5$ & $2^6$ & $2^7$ & $2^8$ & $2^9$ & $2^0$ & $2^1$ \\ 

$2^3$ & $2^4$ & $2^5$ & $2^6$ & $2^7$ & $2^8$ & $2^9$ & $2^0$ & $2^1$ & $2^2$ \\
 
\end{tabular}\]

We arrange these columns in a circle as follows. 
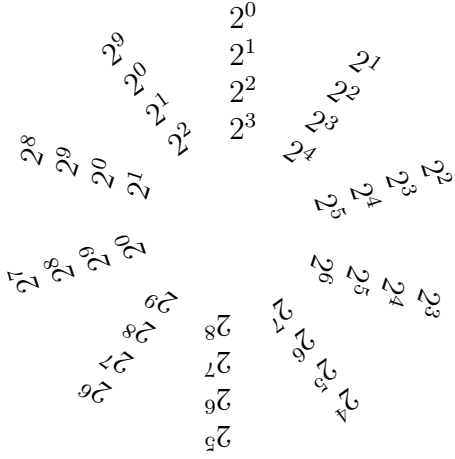
\begin{center}
\begin{tikzpicture}
\foreach \a in {0,1,...,9}{
\FPeval{\angl}{(-360)*\a/10}
\draw (90-\a*360/10: 2cm) node{\rotatebox{\angl}{

\FPeval{\moda}{round((\a+1)-10*\floor{(\a+1)/10}, 0)}
\FPeval{\modb}{round((\a+2)-10*\floor{(\a+2)/10}, 0)}
\FPeval{\modc}{round((\a+3)-10*\floor{(\a+3)/10}, 0)}

$\begin{tabular}{c}

$2^{\a}$ \\ 

$2^{\moda}$ \\ 

$2^{\modb}$ \\ 

$2^{\modc}$
 
\end{tabular}$

}};
}
\end{tikzpicture}
\end{center}

Note that choosing a tuple $(x_1,x_2,\dots, x_n) \in \Z_q^n$ corresponds to choosing $n$ of these columns (then $x_i$ is the topmost element of the $i$th column chosen). Now from the definition of the set $S$, if $(x_1,x_2,\dots,x_n)\in S$, we must choose these columns such that the topmost ($0$th) element of each column does not coincide with any element of any other chosen column. From the circular arrangement of the columns, we see that if we choose a column $C$, then we have to discard exactly $m$ columns to its right and $m$ columns to its left. Hence we simply need to find the number of ways to choose $n$ objects out of $q-1$ distinct circularly placed objects so that between any two chosen objects there are at least $m$ objects which are not chosen. Let us first fix the first element $x_1=2^0$, so we choose the column with topmost element $2^0$. Now label the columns $C_1,C_2,\dots, C_{q-1}$, starting from the column with topmost element $2^0$ and going clockwise. Suppose we choose the columns $C_{i_1},C_{i_2},\dots, C_{i_n}$, with $i_1<i_2<\dots<i_n$. Let $y_j$ be the number of columns between $C_{i_j}$ and $C_{i_{j+1}}$ for $j=0,\dots, n-1$. Then $(y_0,y_1,\dots, y_{n-1})$ determines the chosen columns. We find the number of such tuples in the standard way. We have 
\[ \sum_{j=0}^{n-1}y_j=q-1-n \] where $y_j\geq m$ for each $j$. So further define $z_j=y_j-m$ for $j\geq 0$. So $z_j\geq 0$ for each $j$ and \[ \sum_{j=0}^{n-1}z_j=q-1-n-nm. \] The number of tuples $(z_0,z_1,\dots, z_{n-1})$ determines the number of tuples $(y_0,y_1,\dots,y_{n-1})$ and this number is \[ \binom{q-1-n-nm+n-1}{n-1}=\binom{q-nm-2}{n-1}. \]
Now, we initially assumed that $x_1=2^0$. We can change this starting position in $q-1$ ways (without disturbing the cyclic order yet). Note that this gives us all the $n$ cyclic permutations of a choice $(x_1,x_2,\dots,x_n)$. Finally, there are $\frac{n!}{n}=(n-1)!$ non-cyclic permutations of each choice $(x_1,x_2,\dots,x_n)$, so the final count is \[ \binom{q-nm-2}{n-1}(q-1)(n-1)!=(q-2)(q-nm-2)(q-nm-3)\dots (q-nm-n). \] \par
Since this expression for $\chi_{\A_n^{(m)}}(q)$ holds for infinitely many primes $q$, and $\chi_{\A_n^{(m)}}$ is a polynomial, it follows that \[ \chi_{\A_n^{(m)}}(t) = (t-1)(t-mn-2)(t-mn-3)\dots (t-mn-n). \]

\end{proof}

\begin{corollary}\label{numregions}
The number of regions of the arrangement $\A_n^{(m)}$ is $\frac{2(nm+n+1)!}{(nm+2)!}$.
\end{corollary}

\begin{proof}
By \Cref{zaslavsky}, we have 
\begin{align*}
r(\A_n^{(m)})&=(-1)^n\chi_{\A_n^{(m)}}(-1)\\
&=(-1)^n(-2)(-mn-3)(-mn-4)\dots (-mn-n-1)\\
&=(-1)^n(-1)^n2\frac{(nm+n+1)!}{(nm+2)!}\\
&=\frac{2(nm+n+1)!}{(nm+2)!}.
\end{align*}
\end{proof}

\begin{remark}\label{numregionscatalan}
Let $r_0(\A_n^{(m)})$ denote the number of regions corresponding to coordinates $(x_1,x_2,\dots, x_n)$ with $x_1<x_2<\dots<x_n$. Then $r(\A_n^{(m)})=n!\cdot r_0(\A_n^{(m)})$ and hence 
\begin{align*}
r_0(\A_n^{(m)})&=\frac{2(nm+n+1)!}{n!(nm+2)!}\\
&=\frac{2}{nm+n+2}\binom{nm+n+2}{n}\\
&=\frac{2}{n(m+1)+2}\binom{n(m+1)+2}{n}\\
&=A_n(m,2)
\end{align*}
which are exactly the two parameter Fuss-Catalan numbers mentioned in \Cref{sec1} when $r=2$. Hence \[ r(\A_n^{(m)})=n!\cdot A_n(m,2). \]
\end{remark}

The characteristic polynomials and number of regions of $\A_n^{(m)}$ for some prototypical values of $m,n$ have been listed in \Cref{charpoly}.

\begin{table}[h!]
\captionsetup{font=tiny,labelfont=bf}
\begin{tabular}{|c|c|c|c|}
\hline
$\bm{n}$ & $\bm{m}$ & $\bm{\chi_{\A_n^{(m)}}(t)}$ & $\bm{r(\A_n^{(m)})}$\\
\hline
2 & 1 & $t^2-5t+4$ & 10\\
\hline
2 & 2 & $t^2-7t+6$ &  14\\
\hline
3 & 1 & $t^3-12t^2+41t-30$ & 84\\
\hline
3 & 2 & $t^3-18t^2+89t-72$ & 180\\
\hline
3 & 3 & $t^3-24t^2+155t-132$ & 312\\
\hline
4 & 1 & $t^4-22t^3+167t^2-482t+336$ & 1008\\
\hline
4 & 2 & $t^4-34t^3+395t^2-1682t+1320$ & 3432\\
\hline
4 & 3 & $t^4-46t^3+719t^2-4034t+3360$ & 8160\\
\hline
4 & 4 & $t^4-58t^3+1139t^2-7922t+6840$  & 15960\\
\hline
\end{tabular}
\caption{Characteristic polynomials and number of regions of $\A_n^{(m)}$ for $1\leq m\leq n\leq 4$ ($n>1$)}
\label{charpoly}

\end{table}

The count suggests that the arrangements $\A_n^{(m)}$ are related to the generalised Catalan arrangements $\C_n^{(m)}=\{x_i-x_j=k\mid k\in [-m,m], 1\leq i<j\leq n \}$ which have been studied in detail in \cite[Section 8]{ber}. The number of regions of $\C_n^{(m)}$ has been shown to be equal to $\frac{((m+1)n)!}{(mn+1)!}=n!\cdot \frac{1}{(m+1)n+1}\binom{(m+1)n+1}{n}=n!\cdot A_n(m,1)$. Let $\chi_{\C_n^{(m)}}$ be the characteristic polynomial for the arrangement $\C_n^{(m)}$. We will exhibit a relationship between $\chi_{\C_n^{(m)}}$ and $\chi_{\A_n^{(m)}}$. To do this, we first prove an even more general result. \par
Let $S$ be a finite set of integers. We define the arrangement in $\R^n$ given by
\[
    \A_n^S := \{x_i=0 \mid i \in [n]\} \cup \{x_i=a^kx_j \mid k \in S, 1\leq i<j\leq n\}.
\]
Note that we can show, analogous to the proof of \Cref{nodepa}, that the combinatorics of the arrangement $\A_n^S$ is independent of the value of $a$ as long as $a>1$ is a positive integer.
Similarly, we have the deformation of the braid arrangement in $\R^n$ given by
\[
    \C_n^S := \{x_i-x_j=k \mid k \in S, 1\leq i<j\leq n\}.
\]

\begin{theorem}\label{charpolygeneralised}
For any finite set of integers $S$, we have
\[
    \chi_{\A_n^S}(t) = \chi_{\C_n^S}(t-1).
\]
\end{theorem}

To show this, we will take help of a slightly modified version of \Cref{ffm}, for deformations of the braid arrangement (note that the arrangements $\C_n^{S}$ are such deformations), also given by Athanasiadis in \cite{athanasiadis1999extended}. We state this in our context.

\begin{theorem}[{\cite[Theorem 2.1]{athanasiadis1999extended}}]\label{modffm}
If $\A$ is a deformation of the braid arrangement in $\R^n$, there exists an integer $k$ such that for all integers $q$ greater than $k$, 
\[ \chi(\A,q)= \#\left( \Z_q^n\setminus V_\A  \right)\]
where $V_{\A}$ is the union of the hyperplanes in $\Z_q^n$ obtained by reducing $\A$ mod $q$.
\end{theorem}

\begin{proof}[Proof of \Cref{charpolygeneralised}]
By \Cref{primrootexistence}, we choose $a$ which is a primitive root modulo infinitely many primes. By \Cref{modffm}, we choose $q=p-1$ for a large enough prime $p$ for which $a$ is a primitive root. Now we have $\chi_{\C_n^{S}}(q)=|T|$ where 
\[ T= \{(x_1,x_2,\dots,x_n)\in\Z_q^n\mid x_i-x_j\neq k, k \in S, 1\leq i<j\leq n\}. \] Note that the condition $x_i-x_j\neq k \pmod{q}$ is equivalent to $a^{x_i-x_j}\neq a^k \pmod{p}$, that is, $a^{x_1}\neq a^k\cdot a^{x_j}\pmod{p}$. Now the distinct powers of $a$ with exponents modulo $q$ give distinct (non zero) residues modulo $p$. Hence if we consider the set $T'$ defined as 
\[ T'= \{(y_1,y_2,\dots,y_n)\in\Z_p^n\mid y_i\neq 0, y_i\neq a^ky_j, k \in S, 1\leq i<j\leq n\}, \] then $|T|=|T'|$. Now using \Cref{ffm} on the arrangement $\A_n^S$, we get that $|T'|=\chi_{\A_n^{S}}(p)$, hence $\chi_{\C_n^{S}}(q)=|T|=|T'|=\chi_{A_n^{S}}(p)=\chi_{A_n^{S}}(q+1)$. Since this holds for infinitely many $q=p-1$, and since $\chi_{\C_n^{S}}, \chi_{\A_n^{S}}$ are both polynomials, it follows that for all $t$, 
\[ \chi_{\A_n^{S}}(t)=\chi_{\C_n^{S}}(t-1). \] This completes the proof.
\end{proof}

\begin{remark}
Though we will only make use of \Cref{charpolygeneralised} in the stated form, we note that it can be generalized.
Let $n \geq 1$ and $\mathbf{S}=(S_{i,j})_{1 \leq i < j \leq n}$ be a sequence of finite sets of integers.
To such a tuple $\mathbf{S}$, we associate the arrangement in $\R^n$ given by
\begin{equation*}
    \A_{\mathbf{S}}:=\{x_i = 0\mid 1\leq i\leq n \} \cup \{x_i=a^kx_j \mid 1 \leq i < j \leq n,\ k \in S_{i,j}\}
\end{equation*}
and the deformation of the braid arrangement in $\R^n$ given by
\begin{equation*}
    \C_{\mathbf{S}}:=\{x_i-x_j=k \mid 1 \leq i < j \leq n,\ k \in S_{i,j}\}.
\end{equation*}
Then, just as above, we have
\begin{equation*}
    \chi_{\A_{\mathbf{S}}}(t) = \chi_{\C_{\mathbf{S}}}(t-1).
\end{equation*}
\end{remark}

The following corollary of \Cref{charpolygeneralised} gives us the relationship between the characteristic polynomials of $\A_n^{(m)}$ and $\C_n^{(m)}$.

\begin{corollary}\label{charpolyrel}
We have $\chi_{\A_n^{(m)}}(t)=\chi_{\C_n^{(m)}}(t-1)$.
\end{corollary}

It was already shown by Athanasiadis in \cite[Section 5]{athanas96} that 
\[ \chi_{\C_n^{(m)}}(t)=t(t-mn-1)(t-mn-2)\dots (t-mn-n+1). \]
Hence \Cref{charpolyrel} gives us a way to derive either of $\chi_{\A_n^{(m)}}$ and $\chi_{\C_n^{(m)}}$ from each other.

The combinatorial relationship between $\C_n^{(m)}$ and $\A_n^{(m)}$ will be further detailed in \Cref{sec3}.


\section{Bijections}\label{sec3}

In this section, we exhibit a bijection between the regions of $\A_n^{(m)}$ and certain decorated generalized Dyck paths as well as decorated non-nesting partitions.
From \Cref{nodepa}, we know that the value of $a>1$ in \eqref{arrdef} does not affect the combinatorics of $\A_n^{(m)}$.
For convenience, we now fix $a=2$.

From the definition of the arrangement $\A_n^{(m)}$, we can see that its regions are given by valid total orders on the symbols
\begin{equation*}
    \{0\} \cup \{2^kx_i \mid i \in [n], k \in [0,m]\}.
\end{equation*}
By a \textit{valid} total order, we mean that there is at least one point in $\R^n$ that satisfies the order.

\begin{example}\label{sec3:ex1}
One region of $\A_5^{(2)}$ is given by
\begin{equation*}
    4x_3<2x_3<4x_1<x_3<2x_1<x_1<0<x_5<2x_5<4x_5<x_4<x_2<2x_4<2x_2<4x_4<4x_2.
\end{equation*}
Note that this total order is valid since the point $(-0.5,4,-1.5,3,0.5) \in \R^5$ satisfies the required inequalities.
\end{example}

Akin to \cite[Section 8.1]{ber}, we use the symbols
\begin{equation*}
    T_n^{(m)}:=\{\textcolor{red}{0}\} \cup \{\lt{i}{k} \mid i \in [n], k \in [0,m]\}
\end{equation*}
to represent the total order.
Here, $2^kx_i$ is represented by $\lt{i}{k}$ for $i \in [n]$ and $k \in [0,m]$ and the $0$ in the total order is colored red.

\begin{example}\label{sketchex}
The region of $\A_5^{(2)}$ in \Cref{sec3:ex1} is given by
\begin{equation*}
    \lt{3}{2}\lt{3}{1}\lt{1}{2}\lt{3}{0}\lt{1}{1}\lt{1}{0}\ \textcolor{red}{0}\ \lt{5}{0}\lt{5}{1}\lt{5}{2}\lt{4}{0}\lt{2}{0}\lt{4}{1}\lt{2}{1}\lt{4}{2}\lt{2}{2}.
\end{equation*}
\end{example}

To characterize the valid total orders on $T_n^{(m)}$, we first consider those where $x_i>0$ for all $i \in [n]$.
In this case, considering $X_i:=\operatorname{log}_2(x_i)$, gives a bijection with valid total orders on the symbols
\begin{equation*}
    \{X_i + k \mid i \in [n], k \in [0,m]\}.
\end{equation*}
Such orders have already been characterized; they correspond to the regions of the generalized Catalan arrangement in $\R^n$ given by
\begin{equation*}
    \{X_i-X_j=k \mid k \in [-m,m], 1 \leq i < j \leq n\}.
\end{equation*}
Hence, using the results of \cite[Section 8.1]{ber}, we have the following.

\begin{lemma}\label{allpos}
The regions in $\A_n^{(m)}$ where $x_i>0$ for all $i \in [n]$ correspond to words in $T_n^{(m)}$ where
\begin{enumerate}
    \item each letter of $T_n^{(m)}$ appears exactly once,
    \item the first letter is $\textcolor{red}{0}$,
    \item for any $i,j \in [n]$ and $k,l \in [0,m-1]$, if $\lt{i}{k}$ appears before $\lt{j}{l}$, then $\lt{i}{k+1}$ appears before $\lt{j}{l+1}$, and
    \item for any $i \in [n]$ and $k \in [0,m-1]$, $\lt{i}{k}$ appears before $\lt{i}{k+1}$.
\end{enumerate}
\end{lemma}

Characterizing the other valid total orders on $T_n^{(m)}$ is similar.
Suppose that $B_1,B_2 \subseteq [n]$ are disjoint subsets whose union is $[n]$.
We now characterize the regions of $\A_n^{(m)}$ where $x_i<0$ for all $i \in B_1$ and $x_i>0$ for all $i \in B_2$.

Since $x_i>0$ if and only if $2^kx_i>0$ for all $k \in [0,m]$, we just have to specify separate valid total orders for those letters in $T_n^{(m)}$ which have subscripts in $B_1$ and those that have subscripts in $B_2$.
Setting $Y_i=\operatorname{log}_2(-x_i)$ for $i \in B_1$ and $X_i=\operatorname{log}_2(x_i)$ for $i \in B_2$ gives a bijection between such total orders and a pair consisting of a valid total order on $\{Y_i + k \mid i \in B_1, k \in [0,m]\}$ and one on $\{X_i + k \mid i \in B_2, k \in [0,m]\}$.

Using the above observations and the results of \cite[Section 8.1]{ber}, we have the following characterization of the regions of $\A_n^{(m)}$.

\begin{theorem}\label{regiondesc}
The regions in $\A_n^{(m)}$ where $x_i<0$ for $i \in B_1$ and $x_i>0$ for $i \in B_2$ correspond to words in $T_n^{(m)}$ of the form $w_1\ \textcolor{red}{0}\ w_2$ where
\begin{enumerate}
    \item for $r=1,2$, each letter in $T_n^{(m)}$ of the form $\lt{i}{k}$ where $i \in B_r$ appears exactly once in $w_r$,
    \item the reverse of $w_1$ satisfies properties (3) and (4) of \Cref{allpos}, and
    \item $w_2$ satisfies properties (3) and (4) of \Cref{allpos}.
\end{enumerate}
\end{theorem}

The words in $T_n^{(m)}$ that correspond to regions of $\A_n^{(m)}$, i.e., words of the form mentioned in \Cref{regiondesc}, will henceforth be called $m$-sketches of size $n$.

As mentioned in \cite[Section 8.1]{ber}, the words in $T_n^{(m)}$ of the form mentioned in \Cref{allpos} are in bijection with certain generalized Dyck paths.
By specifying a certain point in such Dyck paths, we can extend this bijection to all regions of $\A_n^{(m)}$.
We now describe this bijection.

\begin{definition}\label{dyckdef}
A labeled $m$-Dyck path of size $n$ is a sequence of $(m+1)n$ terms where
\begin{enumerate}
    \item $n$ terms are `$+m$',
    \item $mn$ terms are `$-1$',
    \item the sum of any prefix of the sequence is non-negative, and
    \item each $+m$ term is given a distinct label from $[n]$.
\end{enumerate}
\end{definition}

A labeled $m$-Dyck path of size $n$ can be drawn in $\mathbb{R}^2$ in the natural way.
Start the path at $(0,0)$, read the labeled $m$-Dyck path and for each term move by $(1,m)$ if it is $+m$, which we call an up-step, and by $(1,-1)$ if it is $-1$, which we call a down-step.
Also, label each up-step with its corresponding label in $[n]$.

\begin{example}
A labeled $2$-Dyck path of size $5$ is given in \Cref{labdyckex}.
\end{example}

\begin{figure}[H]
    \centering
    \begin{tikzpicture}[scale=0.75]
    \draw [thin,lightgray] (0,0)--(15,0);
    \node (0) [circle,inner sep=2pt,fill=black] at (0,0) {};
    \node (1) [circle,inner sep=2pt,fill=black] at (1,2) {};
    \node (2) [circle,inner sep=2pt,fill=black] at (2,1) {};
    \node (3) [circle,inner sep=2pt,fill=black] at (3,3) {};
    \node (4) [circle,inner sep=2pt,fill=black] at (4,2) {};
    \node (5) [circle,inner sep=2pt,fill=black] at (5,1) {};
    \node (6) [circle,inner sep=2pt,fill=black] at (6,0) {};
    \node (7) [circle,inner sep=2pt,fill=black] at (7,2) {};
    \node (8) [circle,inner sep=2pt,fill=black] at (8,1) {};
    \node (9) [circle,inner sep=2pt,fill=black] at (9,0) {};
    \node (10) [circle,inner sep=2pt,fill=black] at (10,2) {};
    \node (11) [circle,inner sep=2pt,fill=black] at (11,4) {};
    \node (12) [circle,inner sep=2pt,fill=black] at (12,3) {};
    \node (13) [circle,inner sep=2pt,fill=black] at (13,2) {};
    \node (14) [circle,inner sep=2pt,fill=black] at (14,1) {};
    \node (15) [circle,inner sep=2pt,fill=black] at (15,0) {};
    \draw (0)-- node[left] {3} (1)--(2)--node[left] {1} (3)--(4)--(5)--(6)--node[left] {5} (7)--(8)--(9)--node[left] {4} (10)--node[left] {2} (11)--(12)--(13)--(14)--(15);
    \end{tikzpicture}
    \caption{A labeled $2$-Dyck path of size $5$.}
    \label{labdyckex}
\end{figure}
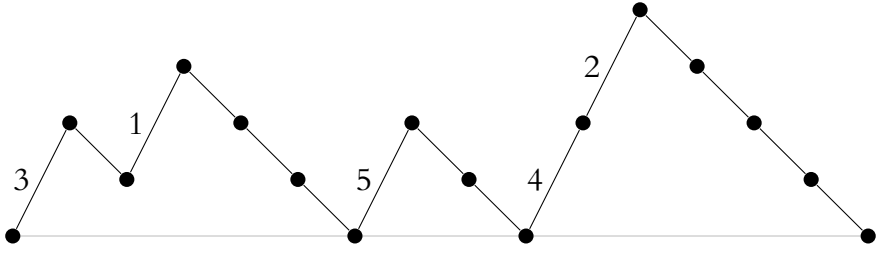


\begin{definition}
A decorated $m$-Dyck path of size $n$ is a labeled $m$-Dyck path of size $n$ where some point on the Dyck path that lies on the $x$-axis is specified.
\end{definition}

We now show that the regions of $\A_n^{(m)}$ are in bijection with decorated $m$-Dyck paths of size $n$.
We first consider the case where $x_i>0$ for all $i \in [n]$.
From \cite[Section 8.1]{ber}, we know that we obtain a labeled $m$-Dyck path from a word in $T_n^{(m)}$ of the form mentioned in \Cref{allpos} by replacing each $\lt{i}{0}$ with an up-step labeled $i$, and all other letters with a down-step.
This is in fact a bijection with labeled $m$-Dyck path of size $n$.
Here we select the point $(0,0)$ to give us the associated decorated Dyck path.

\begin{example}\label{bijecallpos}
The region of $\A_3^{(2)}$ given by
\begin{equation*}
    \textcolor{red}{0}\ \lt{3}{0}\lt{3}{1}\lt{3}{2}\lt{1}{0}\lt{2}{0}\lt{1}{1}\lt{2}{1}\lt{1}{2}\lt{2}{2}
\end{equation*}
has corresponding decorated $2$-Dyck path given in \Cref{allposdyck}.
\end{example}

\begin{figure}[H]
    \centering
    \begin{tikzpicture}[scale=0.75]
    \draw [thin,lightgray] (6,0)--(15,0);
    \node (6) [circle,inner sep=3pt,fill=red] at (6,0) {};
    \node (7) [circle,inner sep=2pt,fill=black] at (7,2) {};
    \node (8) [circle,inner sep=2pt,fill=black] at (8,1) {};
    \node (9) [circle,inner sep=2pt,fill=black] at (9,0) {};
    \node (10) [circle,inner sep=2pt,fill=black] at (10,2) {};
    \node (11) [circle,inner sep=2pt,fill=black] at (11,4) {};
    \node (12) [circle,inner sep=2pt,fill=black] at (12,3) {};
    \node (13) [circle,inner sep=2pt,fill=black] at (13,2) {};
    \node (14) [circle,inner sep=2pt,fill=black] at (14,1) {};
    \node (15) [circle,inner sep=2pt,fill=black] at (15,0) {};
    \draw (6)--node[left] {3} (7)--(8)--(9)--node[left] {1} (10)--node[left] {2} (11)--(12)--(13)--(14)--(15);
    \end{tikzpicture}
    \caption{Decorated Dyck path corresponding to the region of $\A_3^{(2)}$ given in \Cref{bijecallpos}.}
    \label{allposdyck}
\end{figure}
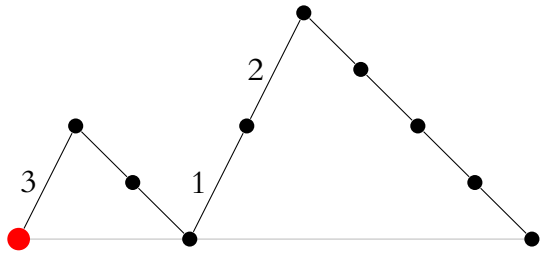

In the general case, we have an $m$-sketch of the form $w_1\ \textcolor{red}{0}\ w_2$.
Here we construct a Dyck path in the same way, except that for $w_1$, we use the letters of the form $\lt{i}{m}$ for the up-steps (recall that the \textit{reverse} of $w_1$ satisfies properties (3) and (4) of \Cref{allpos}).
The point on the $x$-axis that is specified is the one where the Dyck path corresponding to $w_1$ ends.

\begin{example}
The region of $\A_5^{(2)}$ in \Cref{sec3:ex1} has corresponding decorated $2$-Dyck path given in \Cref{dyck}.
\end{example}

\begin{figure}[H]
    \centering
    \begin{tikzpicture}[scale=0.75]
    \draw [thin,lightgray] (0,0)--(15,0);
    \node (0) [circle,inner sep=2pt,fill=black] at (0,0) {};
    \node (1) [circle,inner sep=2pt,fill=black] at (1,2) {};
    \node (2) [circle,inner sep=2pt,fill=black] at (2,1) {};
    \node (3) [circle,inner sep=2pt,fill=black] at (3,3) {};
    \node (4) [circle,inner sep=2pt,fill=black] at (4,2) {};
    \node (5) [circle,inner sep=2pt,fill=black] at (5,1) {};
    \node (6) [circle,inner sep=3pt,fill=red] at (6,0) {};
    \node (7) [circle,inner sep=2pt,fill=black] at (7,2) {};
    \node (8) [circle,inner sep=2pt,fill=black] at (8,1) {};
    \node (9) [circle,inner sep=2pt,fill=black] at (9,0) {};
    \node (10) [circle,inner sep=2pt,fill=black] at (10,2) {};
    \node (11) [circle,inner sep=2pt,fill=black] at (11,4) {};
    \node (12) [circle,inner sep=2pt,fill=black] at (12,3) {};
    \node (13) [circle,inner sep=2pt,fill=black] at (13,2) {};
    \node (14) [circle,inner sep=2pt,fill=black] at (14,1) {};
    \node (15) [circle,inner sep=2pt,fill=black] at (15,0) {};
    \draw (0)-- node[left] {3} (1)--(2)--node[left] {1} (3)--(4)--(5)--(6)--node[left] {5} (7)--(8)--(9)--node[left] {4} (10)--node[left] {2} (11)--(12)--(13)--(14)--(15);
    \end{tikzpicture}
    \caption{Decorated Dyck path corresponding to the region of $\A_5^{(2)}$ given in \Cref{sec3:ex1}.}
    \label{dyck}
\end{figure}

It is also possible to obtain \Cref{numregions} directly from this description of the regions of $\A_n^{(m)}$.
Note that a decorated $m$-Dyck path of size $n$ is equivalent to
\begin{enumerate}
    \item a pair of \textit{unlabeled} $m$-Dyck paths whose total number of up-steps is $n$, and
    \item a permutation of $[n]$.
\end{enumerate}
Using the result from \cite[Example 7.5.5]{conc} that there are $A_k(m,1)$ unlabeled $m$-Dyck paths with $k$ up-steps, we get
\begin{equation*}
    r(\A_n^{(m)}) = n! \cdot \sum_{k=0}^{n}A_k(m,1)A_{n-k}(m,1) = n! \cdot A_n(m,2).
\end{equation*}
The second equality follows from \cite[Equation 7.70]{conc}.

\begin{remark}
In fact, \cite[Equation 7.70]{conc} shows that $A_n(m,r)$ is the number of $r$-tuples of (possibly empty) unlabeled $m$-Dyck paths whose total number of up-steps is $n$.
\end{remark}

Also, using the fact that there are $A_{n-k}(m,mk)$ unlabeled $m$-Dyck paths with $n$ up-steps and $(k+1)$ points on the $x$-axis (see \cite{rethill}), we get the following result.

\begin{result}
For any $n,m \geq 1$, we have
\begin{equation*}
    r(\A_n^{(m)}) = n! \cdot \sum_{k=1}^{n}(k+1)A_{n-k}(m,mk).
\end{equation*}
\end{result}

We now describe the bijection between the regions of $\A_n^{(m)}$ and certain decorated non-nesting partitions, which will be useful in \Cref{subarrsec}.


For any $n \geq 1$, partitions of $[n]$ can be expressed by placing $n$ dots in a row to represent the numbers $1,2,\ldots,n$ in order and drawing arcs to represent the blocks.
If the numbers $i_1 < i_2 < \cdots < i_k$ form a block, then an arc is drawn from $i_j$ to $i_{j+1}$ for all $j \in [k-1]$.
In this case, the partitions where there are no nesting arcs are called \textit{non-nesting partitions}.

\begin{example}
The arc diagrams for the partitions
\begin{enumerate}
    \item $\Pi_1 = \{\{1\} ,\{2,4\},\{3,5,6\}\}$ and
    \item $\Pi_2 = \{\{1,4,5\},\{2\},\{3,6\}\}$
\end{enumerate}
are drawn in \Cref{nonnest1,nonnest2} respectively.
Only $\Pi_1$ is non-nesting.
\end{example}

\begin{figure}[H]
    \centering
    \begin{tikzpicture}
        \node[circle,fill=black,inner sep =2pt] (-6) at (-6+0.5,0) {};
        \node[circle,fill=black,inner sep =2pt] (-5) at (-5+0.5,0) {};
        \node[circle,fill=black,inner sep =2pt] (-4) at (-4+0.5,0) {};
        \node[circle,fill=black,inner sep =2pt] (-3) at (-3+0.5,0) {};
        \node[circle,fill=black,inner sep =2pt] (-2) at (-2+0.5,0) {};
        \node[circle,fill=black,inner sep =2pt] (-1) at (-1+0.45,0) {};
        
        \draw (-5.north)..controls +(up:13mm) and +(up:13mm)..(-3.north);
        \draw (-4.north)..controls +(up:13mm) and +(up:13mm)..(-2.north);
        \draw (-2.north)..controls +(up:7mm) and +(up:7mm)..(-1.north);
    \end{tikzpicture}
    \caption{Arc diagram for the partition $\{\{1\},\{2,4\},\{3,5,6\}\}$.}
    \label{nonnest1}
\end{figure}
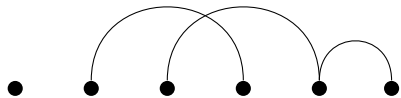

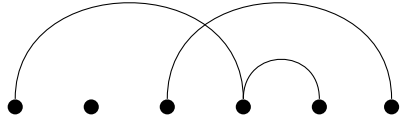
\begin{figure}[H]
    \centering
    \begin{tikzpicture}
        \node[circle,fill=black,inner sep =2pt] (-6) at (-6+0.5,0) {};
        \node[circle,fill=black,inner sep =2pt] (-5) at (-5+0.5,0) {};
        \node[circle,fill=black,inner sep =2pt] (-4) at (-4+0.5,0) {};
        \node[circle,fill=black,inner sep =2pt] (-3) at (-3+0.5,0) {};
        \node[circle,fill=black,inner sep =2pt] (-2) at (-2+0.5,0) {};
        \node[circle,fill=black,inner sep =2pt] (-1) at (-1+0.45,0) {};
        
        \draw (-6.north)..controls +(up:17mm) and +(up:17mm)..(-3.north);
        \draw (-3.north)..controls +(up:7mm) and +(up:7mm)..(-2.north);
        \draw (-4.north)..controls +(up:17mm) and +(up:17mm)..(-1.north);
    \end{tikzpicture}
    \caption{Arc diagram for the partition $\{\{1,4,5\},\{2\},\{3,6\}\}$.}
    \label{nonnest2}
\end{figure}

We now define the decorated non-nesting partitions that correspond to the regions of $\A_n^{(m)}$.

\begin{definition}
A decorated $m$-non-nesting partition of size $n$ is an ordered pair of (possibly empty) non-nesting partitions such that
\begin{enumerate}
    \item the size of each block is $(m+1)$,
    \item the total number of blocks is $n$, and
    \item each block is given a distinct label from $[n]$.
\end{enumerate}
\end{definition}

Just as for regular partitions, decorated non-nesting partitions can be represented by arc diagrams.
If the pair of partitions is $(\Pi_1,\Pi_2)$, we draw the arc diagram for $\Pi_1$ followed by a red line and then the arc diagram for $\Pi_2$.
We represent the label of a block by replacing the dot corresponding to each number in that block by the label.

\begin{example}\label{decnonnestex}
The decorated $2$-non-nesting partition of size $5$ given by $(\Pi_1,\Pi_2)$ where
\begin{enumerate}
    \item $\Pi_1=\{\{1,2,4\},\{3,5,6\}\}$ with first block labeled $3$ and second labeled $1$, and
    \item $\Pi_2=\{\{1,2,3\},\{4,6,8\},\{5,7,9\}\}$ with first block labeled $5$, second labeled $4$, and third labeled $2$
\end{enumerate}
has corresponding arc diagram given in \Cref{decnonnest}.
\end{example}

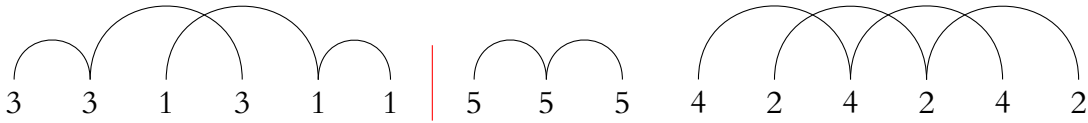
\begin{figure}[H]
    \centering
    \begin{tikzpicture}
        \node (-6) at (-6+0.5,0) {3};
        \node (-5) at (-5+0.5,0) {3};
        \node (-4) at (-4+0.5,0) {1};
        \node (-3) at (-3+0.5,0) {3};
        \node (-2) at (-2+0.5,0) {1};
        \node (-1) at (-1+0.45,0) {1};
        
        \draw [red] (0,-0.25)--(0,0.75);
        
        \node (1) at (1-0.45,0) {5};
        \node (2) at (2-0.5,0) {5};
        \node (3) at (3-0.5,0) {5};
        \node (4) at (4-0.5,0) {4};
        \node (5) at (5-0.5,0) {2};
        \node (6) at (6-0.5,0) {4};
        \node (7) at (7-0.5,0) {2};
        \node (8) at (8-0.5,0) {4};
        \node (9) at (9-0.5,0) {2};
        
        \draw (-6.north)..controls +(up:7mm) and +(up:7mm)..(-5.north);
        \draw (-5.north)..controls +(up:13mm) and +(up:13mm)..(-3.north);
        \draw (-4.north)..controls +(up:13mm) and +(up:13mm)..(-2.north);
        \draw (-2.north)..controls +(up:7mm) and +(up:7mm)..(-1.north);
        
        \draw (1.north)..controls +(up:7mm) and +(up:7mm)..(2.north);
        \draw (2.north)..controls +(up:7mm) and +(up:7mm)..(3.north);
        \draw (4.north)..controls +(up:13mm) and +(up:13mm)..(6.north);
        \draw (6.north)..controls +(up:13mm) and +(up:13mm)..(8.north);
        \draw (5.north)..controls +(up:13mm) and +(up:13mm)..(7.north);
        \draw (7.north)..controls +(up:13mm) and +(up:13mm)..(9.north);
    \end{tikzpicture}
    \caption{Decorated non-nesting partition given in \Cref{decnonnestex}.}
    \label{decnonnest}
\end{figure}

The bijection between $m$-sketches of size $n$ (which we know correspond to regions of $\A_n^{(m)}$) and decorated $m$-non-nesting partitions of size $n$ is fairly straightforward.
Each letter $\lt{i}{k}$ in a sketch is replaced by the label $i$ and $\textcolor{red}{0}$ is replaced by a red line.
For each $i \in [n]$, an arc is drawn joining any occurrence of $i$ with the subsequent occurrence.
Property (1) of \Cref{regiondesc} ensures that there are $n$ blocks labeled distinctly using $[n]$ and that each block is of size $(m+1)$.
Properties (2) and (3) of \Cref{regiondesc} ensure that there is no nesting of arcs.

\begin{example}
The $2$-sketch given in \Cref{sketchex} has corresponding decorated $2$-non-nesting partition given in \Cref{decnonnest}.
\end{example}

\subsection{Sub-arrangements of $\A_n^{(m)}$}\label{subarrsec}

As mentioned in \Cref{sec1}, the arrangement $\A_n^{(m)}$ is a generalization of the arrangement $\alpha_n$ in \cite{seo}.
We now count the regions for generalizations of the other arrangements mentioned in \cite{seo}.

We first consider the arrangement in $\R^n$ given by
\begin{equation*}
    \B_n^{(m)}:=\{x_i=2^kx_j \mid k \in [0,m]\text{ and }i,j \in [n], i \neq j\}.
\end{equation*}
This is a generalization of the arrangement $\beta_n(=\B_n^{(1)})$ considered in \cite{seo}.

The arrangement $\B_n^{(m)}$ is obtained by removing the coordinate hyperplanes from $\A_n^{(m)}$.
Since $\B_n^{(m)}$ is a sub-arrangement of $\A_n^{(m)}$, each region of $\B_n^{(m)}$ breaks up into regions of $\A_n^{(m)}$.
Just as in \cite[Section 8.2]{ber}, we count the regions of $\B_n^{(m)}$ by choosing a canonical region of $\A_n^{(m)}$ from each region of $\B_n^{(m)}$.

To do this, we first have to characterize which regions of $\A_n^{(m)}$ lie in the same region of $\B_n^{(m)}$.
We do this in terms of decorated non-nesting partitions.

\begin{definition}
A block in a decorated $m$-non-nesting partition is called isolated if the $(m+1)$ points that constitute it are consecutive.
Otherwise the block is said to be tangled.
\end{definition}

\begin{example}
The only isolated block in the decorated partition in \Cref{decnonnest} is the one labeled $5$.
\end{example}

In the following proposition, we identify regions of $\A_n^{(m)}$ with $m$-non-nesting partitions of size $n$.
Also, we use arc diagrams to represent decorated non-nesting partitions.

\begin{proposition}\label{bequiv}
Any two decorated $m$-non-nesting partitions of size $n$ lie in the same region of $\B_n^{(m)}$ if and only if
\begin{enumerate}
    \item they at most differ by the position of the red line, and
    \item the position of the red line and any tangled block is the same.
\end{enumerate}
\end{proposition}
\begin{proof}
The positions of the blocks relative to each other describe inequalities corresponding to hyperplanes of the form $x_i=2^kx_j$ for some $k \in [0,m]$ and distinct $i,j \in [n]$.
Hence any two decorated non-nesting partitions in the same region of $\B_n^{(m)}$ must result in the same labeled arc diagram when the red line is deleted.

Next, suppose that we are given a decorated non-nesting partition with a tangled block after the red line.
Suppose the label of the block is $i \in [n]$.
Consider a decorated non-nesting partition obtained by changing the position of the red line so that the block labeled $i$ is before the red line.
We show that these two decorated non-nesting partitions cannot be in the same region of $\B_n^{(m)}$.

Since $i$ is a tangled block, there is some block $j$ with a point between points in the block $i$.
Suppose this corresponds to the inequality $2^kx_i<2^lx_j<2^{k+1}x_i$ for some $k \in [0,m-1]$ and $l \in [0,m]$.
However, when the red line is after the block labeled $i$, we have $2^{k+1}x_i<2^kx_i$ and hence cannot have $2^kx_i<2^lx_j<2^{k+1}x_i$.

This shows that any two decorated non-nesting partitions in the same region of $\B_n^{(m)}$ must satisfy properties (1) and (2) of the statement.
Similarly, it can be shown that any two decorated non-nesting partitions satisfying properties (1) and (2) lie of the same side of hyperplanes of the form $x_i=2^kx_j$ for all $k \in [0,m]$ and distinct $i,j \in [n]$ and hence lie in the same region of $\B_n^{(m)}$.
\end{proof}

\begin{example}
The only decorated non-nesting partition that lies in the same region of $\B_5^{(2)}$ as the one in \Cref{decnonnest} is given in \Cref{decbequiv}.
\end{example}

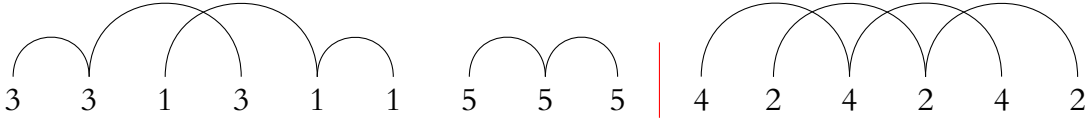
\begin{figure}[H]
    \centering
    \begin{tikzpicture}
        \node (-6) at (-6+0.5,0) {3};
        \node (-5) at (-5+0.5,0) {3};
        \node (-4) at (-4+0.5,0) {1};
        \node (-3) at (-3+0.5,0) {3};
        \node (-2) at (-2+0.5,0) {1};
        \node (-1) at (-1+0.5,0) {1};
        \node (1) at (0+0.5,0) {5};
        \node (2) at (1+0.5,0) {5};
        \node (3) at (2+0.45,0) {5};
        
        \draw [red] (3,-0.25)--(3,0.75);
        
        \node (4) at (4-0.45,0) {4};
        \node (5) at (5-0.5,0) {2};
        \node (6) at (6-0.5,0) {4};
        \node (7) at (7-0.5,0) {2};
        \node (8) at (8-0.5,0) {4};
        \node (9) at (9-0.5,0) {2};
        
        \draw (-6.north)..controls +(up:7mm) and +(up:7mm)..(-5.north);
        \draw (-5.north)..controls +(up:13mm) and +(up:13mm)..(-3.north);
        \draw (-4.north)..controls +(up:13mm) and +(up:13mm)..(-2.north);
        \draw (-2.north)..controls +(up:7mm) and +(up:7mm)..(-1.north);
        
        \draw (1.north)..controls +(up:7mm) and +(up:7mm)..(2.north);
        \draw (2.north)..controls +(up:7mm) and +(up:7mm)..(3.north);
        \draw (4.north)..controls +(up:13mm) and +(up:13mm)..(6.north);
        \draw (6.north)..controls +(up:13mm) and +(up:13mm)..(8.north);
        \draw (5.north)..controls +(up:13mm) and +(up:13mm)..(7.north);
        \draw (7.north)..controls +(up:13mm) and +(up:13mm)..(9.north);
    \end{tikzpicture}
    \caption{Decorated non-nesting partition lying in the same region of $\B_5^{(2)}$ as \Cref{decnonnest}.}
    \label{decbequiv}
\end{figure}

Using \Cref{bequiv}, we can see that each region of $\B_n^{(m)}$ has exactly one region of $\A_n^{(m)}$ whose corresponding decorated non-nesting partition has its red line \textit{not} followed by an isolated block.
Counting such decorated non-nesting partitions gives us the following result, which generalizes \cite[Corollary 3.6]{seo}.

\begin{result}
For any $n,m \geq 1$, we have
\begin{equation*}
    r(\B_n^{(m)}) = n! \cdot (A_n(m,2) - A_{n-1}(m,2)).
\end{equation*}
\end{result}
\begin{proof}
For any $j \geq 1$, there are $A_j(m+1,1)$ non-nesting partitions on $[(m+1)j]$ where each block is of size $(m+1)$.
Hence, the number of $m$-decorated non-nesting partitions of size $n$ where there is an isolated block following the red line is
\begin{equation*}
    n! \cdot \sum_{k=1}^{n}A_{n-k}(m,1)A_{k-1}(m,1) = n! \cdot A_{n-1}(m,2).
\end{equation*}
The result now follows from \Cref{numregionscatalan}.
\end{proof}

Similarly, we consider generalizations of the other arrangements studied in \cite{seo}.

\begin{enumerate}
    \item The arrangement in $\R^n$ given by
    \begin{equation*}
        \Gamma_n^{(m)}:= \A_n^{(m)} \setminus \{x_j=2^mx_i \mid 1 \leq i < j \leq n\}
    \end{equation*}
    is a generalization of the arrangement $\gamma_n$ ($=\Gamma_n^{(1)}$) considered in \cite{seo}.
    
    \item The arrangement in $\R^n$ given by
    \begin{equation*}
        \Delta_n^{(m)}:= \Gamma_n^{(m)} \setminus \{x_i=0 \mid i \in [n]\}
    \end{equation*}
    is a generalization of the arrangement $\delta_n$ ($=\Delta_n^{(1)}$) considered in \cite{seo}.
\end{enumerate}

Again, using techniques from \cite[Section 8.2]{ber}, we can count the regions of these sub-arrangements of $\A_n^{(m)}$.
We state these results without proof.
The enumeration of the regions of $\Gamma_n^{(m)}$ is similar to that of the $m$-Shi arrangement mentioned in \cite{ber}.
The value of $r(\Delta_n^{(m)})$ can be obtained from the one for $r(\Gamma_n^{(m)})$ just as $r(\B_n^{(m)})$ was obtained from $r(\A_n^{(m)})$.

\begin{result}
For any $n,m \geq 1$, we have
\begin{equation*}
    r(\Gamma_n^{(m)}) = \sum_{k=0}^n \binom{n}{k}(mk+1)^{k-1}(m(n-k)+1)^{n-k-1}.
\end{equation*}
\end{result}

\begin{result}
For any $n,m \geq 1$, we have
\begin{equation*}
    r(\Delta_n^{(m)}) = r(\Gamma_n^{(m)}) - \sum_{k=1}^n \binom{n}{k}k(m(k-1)+1)^{k-2}(m(n-k)+1)^{n-k-1}.
\end{equation*}
\end{result}

The ideas that we used to relate the regions of $\A_n^{(m)}$ to those of $\C_n^{(m)}$ can be generalized to obtain the following theorem.

\begin{theorem}
For any finite set of integers $S$, we have for any $n \geq 0$,
\begin{equation*}
    r(\A_n^S) = \sum_{k=0}^n \binom{n}{k} r(\C_k^S)r(\C_{n-k}^S).
\end{equation*}
\end{theorem}

\begin{remark}
Bernardi \cite[Theorem 8.8]{ber} has shown that if the set $S$ satisfies certain conditions (see \cite[Definition 3.5]{ber}), the regions of $\C_n^S$ are in bijection with certain decorated non-nesting partitions.
Just as in the case of $\A_n^{(m)}$, this bijection can be extended to one for the regions of $\A_n^S$ by considering pairs of non-nesting partitions: one for the positive coordinates and one for the negative.
\end{remark}


\section{A statistic on decorated Dyck paths}\label{sec4}

In this section, we define a statistic on decorated $m$-Dyck paths of size $n$ whose distribution is given by the coefficients of the characteristic polynomial of $\A_n^{(m)}$.

From \Cref{charpolyrel}, we know that $\chi_{\A_n^{(m)}}(t)=\chi_{\C_n^{(m)}}(t-1)$.
This allows us to use the interpretation of the coefficients of $\chi_{\C_n^{(m)}}(t)$  from \cite{branch} to describe our required statistic.
We first recall this interpretation.

We slightly extend the definition of a labeled $m$-Dyck path of size $n$ (see \Cref{dyckdef}) by omitting the condition that the labels on the up-steps should be $[n]$.
However, we still assume that the labels are distinct positive integers.

A labeled Dyck path breaks up into \textit{primitive} parts based on when it touches the $x$-axis.
If a labeled Dyck path has $k$ primitive parts with label set $S$, then we break the path into \textit{compartments} as follows.
If the number largest label is in the $i_1^{th}$ primitive part, then the primitive parts up to the $i_1^{th}$ form the first compartment.
Let $j$ be the largest number in $S \setminus A$ where $A$ is the set of labels in first compartment.
If $j$ is in the $i_2^{th}$ primitive part then the primitive parts after the $i_1^{th}$ up to the $i_2^{th}$ form the second compartment.
Continuing this way, we break up a labeled Dyck path into compartments.

\begin{example}
The labeled $1$-Dyck path given in \Cref{dyckcomp} has $3$ primitive parts and $2$ compartments.
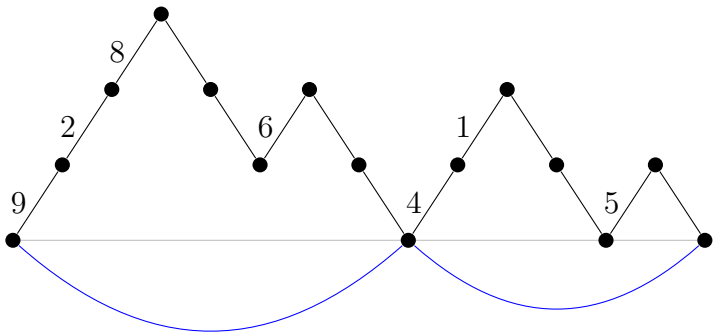
\begin{figure}[!htbp]
    \centering
    \begin{tikzpicture}[xscale=0.65]
    \draw [thin,lightgray] (8,0)--(22,0);
    \node (8) [circle,inner sep=2pt,fill=black] at (8,0) {};
    \node (9) [circle,inner sep=2pt,fill=black] at (9,1) {};
    \node (10) [circle,inner sep=2pt,fill=black] at (10,2) {};
    \node (11) [circle,inner sep=2pt,fill=black] at (11,3) {};
    \node (12) [circle,inner sep=2pt,fill=black] at (12,2) {};
    \node (13) [circle,inner sep=2pt,fill=black] at (13,1) {};
    \node (14) [circle,inner sep=2pt,fill=black] at (14,2) {};
    \node (15) [circle,inner sep=2pt,fill=black] at (15,1) {};
    \node (16) [circle,inner sep=2pt,fill=black] at (16,0) {};
    \node (17) [circle,inner sep=2pt,fill=black] at (17,1) {};
    \node (18) [circle,inner sep=2pt,fill=black] at (18,2) {};
    \node (19) [circle,inner sep=2pt,fill=black] at (19,1) {};
    \node (20) [circle,inner sep=2pt,fill=black] at (20,0) {};
    \node (21) [circle,inner sep=2pt,fill=black] at (21,1) {};
    \node (22) [circle,inner sep=2pt,fill=black] at (22,0) {};
    \draw (8)--node[left] {$9$}(9)--node[left] {$2$}(10)--node[left] {$8$}(11)--(12)--(13)--node[left] {$6$}(14)--(15)--(16)--node[left] {$4$}(17)--node[left] {$1$}(18)--(19)--(20)--node[left] {$5$}(21)--(22);
    \draw [blue,edgee] (8) to [bend right] (16);
    \draw [blue,edgee] (16) to [bend right] (22);
    \end{tikzpicture}
    \caption{A labeled $1$-Dyck path with compartments specified.}
    \label{dyckcomp}
\end{figure}
\end{example}

It can be checked that knowing the compartments (without the order in which they appear) is enough to reconstruct a labeled Dyck path.
Hence they form ``connected components'' of the labeled Dyck path.
This means that a labeled Dyck path with label set $S$ is specified by a partition of $S$ along with a connected Dyck path structure (Dyck path with one compartment) on each block of the partition.

We have the following result from \cite{branch}.

\begin{theorem}[{\cite[Result 4.13]{branch}}]
The absolute value of the coefficient of $t^j$ in $\chi_{\C_n^{(m)}}(t)$ is the number of labeled $m$-Dyck paths with label set $[n]$ which have $j$ compartments.
\end{theorem}

We now turn to decorated Dyck paths.
Just as for non-nesting partitions, a decorated $m$-Dyck path of size $n$ can be seen as an ordered pair of labeled $m$-Dyck paths having disjoint label sets whose union is $[n]$.
From the discussion above, this is the same as a collection of connected labeled Dyck paths having disjoint labels sets whose union is $[n]$ along with a choice of some of these Dyck paths.
This can be seen by breaking up each labeled Dyck path in the pair into connected components and specifying which ones are from the second Dyck path.

We now describe the required statistic on decorated Dyck paths.
In the following theorem, we view decorated Dyck paths as a pair of labeled Dyck paths.

\begin{theorem}
The absolute value of the coefficient of $t^j$ in $\chi_{\A_n^{(m)}}(t)$ is the number of decorated $m$-Dyck paths of size $n$ for which the second labeled Dyck path has $j$ compartments.
\end{theorem}
\begin{proof}
We set $C(m,n,j)$ to be the absolute value of the coefficient of $t^j$ in $\chi_{\C_n^{(m)}}(t)$ and $P(m,n,j)$ to be the absolute value of the coefficient of $t^j$ in $\chi_{\A_n^{(m)}}(t)$.
From \Cref{charpolyrel} and the fact that the coefficients of the characteristic polynomial of any arrangement alternate in sign (see \cite[Corollary 3.4]{stanarr07}), we get
\begin{equation*}
    P(m,n,j) = \sum_{i=j}^n C(m,n,i)\binom{i}{j}.
\end{equation*}
The result now follows from the alternate description of decorated Dyck paths mentioned before the statement of the theorem.
\end{proof}

\begin{example}
The second labeled Dyck path in the decorated Dyck path given in \Cref{dyck} has $2$ compartments.
\end{example}

\begin{remark}
If $S$ is a transitive set (see \cite[Definition 3.5]{ber}) an interpretation of the coefficients of $\chi_{\C_n^S}(t)$ is given in \cite[Theorem 3.5]{branch}.
Just as above, this can be used along with \Cref{charpolygeneralised} to obtain an interpretation for the coefficients of $\chi_{\A_n^S}(t)$.
\end{remark}


\end{document}